\documentclass[12pt]{article}
\usepackage{amssymb}	
\usepackage{amsmath}    
\usepackage{epsfig}
\usepackage{multirow}
\usepackage{longtable}

\newenvironment{proof}[1][Proof:]{\begin{trivlist} 
\item[\hskip \labelsep {\bfseries #1}]}{\end{trivlist}} 
\newcommand{\qed}{\nobreak \ifvmode \relax \else \ifdim\lastskip<1.5em \hskip-\lastskip \hskip1.5em plus0em minus0.5em \fi \nobreak \vrule height0.75em width0.5em depth0.25em\fi} 
\def\0{\bf \0}
\def\A{{\bf A}}

\def\I{{\bf I}}

\def\M{{\bf M}}

\def\0{{\bf 0}}

\def\R{\mathbb{R}}
\def\S{{\bf S}}
\def\T{{\bf T}}

\def\Z{{\bf Z}}
\def\a{{\bf a}}
\def\b{{\bf b}}
\def\c{{\bf c}}

\def\n{{\bf n}}

\def\s{{\bf s}}
\def\t{{\bf t}}

\def\x{{\bf x}}
\def\y{{\bf y}}

\def\Tr{{\rm T}}
\def\T{{\rm T}}

\newtheorem{algorithm}{Algorithm}[section]

\newtheorem{theorem}{Theorem}[section]
\newtheorem{lemma}{Lemma}[section]

\newtheorem{corollary}{Corollary}[section]
\newtheorem{remark}{Remark}[section]

\usepackage{color}
\usepackage[normalem]{ulem}

\begin{document}
\title{A double-pivot simplex algorithm and its upper 
bounds of the iteration numbers}
\author{
Yaguang Yang\thanks{US NRC, Office of Research, 
11555 Rockville Pike, Rockville, 20850. 
Email: yaguang.yang@verizon.net.} 
}

\date{\today}

\maketitle    

\begin{abstract}

In this paper, a double-pivot simplex method is proposed.
Two upper bounds of iteration numbers are derived. 
Applying one of the bounds to some special linear 
programming (LP) problems, such as LP with a totally
unimodular matrix and Markov Decision Problem (MDP)
with a fixed discount rate, indicates that the double-pivot 
simplex method solves these problems in a strongly 
polynomial time. Applying the other bound to a variant of 
Klee-Minty cube shows that this bound
is actually attainable. Numerical test on three variants of 
Klee-Minty cubes is performed for the problems with 
sizes as big as $200$ constraints and $400$ variables.
The test result shows that the proposed  algorithm 
performs extremely good for all three variants. 
Dantzig's simplex method cannot handle the Klee-Minty 
cube problems with $200$ constraints because it needs
about $2^{200} \approx 10^{60}$ iterations. 
Numerical test is also performed for randomly 
generated problems for both the proposed and
Dantzig's simplex methods. This test shows that 
the proposed method is promising for large 
size problems.

\end{abstract}

{\bf Keywords:} Double-pivot algorithm, simplex method, 
linear programming, Klee-Minty cube.

{\bf MSC classification:} 90C05 90C49.
 \newpage
 
\section{Introduction}

Since Dantzig invented the simplex method in 1940s 
\cite{dantzig49}, its complexity has been a topic attracted
many researchers. Since all pivot rules of the
simplex method search the optimizer among 
vertices which are defined by the linear 
constraints, the iterate moves from one vertex to the next
vertex along an edge of the polytope. Therefore, the
diameter of a polytope, defined as the shortest path or 
the least number of edges between any two vertices of 
the polytope, is the smallest iteration number that 
the best simplex algorithm can possibly achieve. Hirsch 
in 1957 \cite{dantzig63} conjectured that the diameter 
of the polytope is $m-n$ for 
the polytope $P=\{ \x\in \R^n: \A\x \le \b \}$ where 
$\A \in \Z^{m \times n}$ and $m>n$. This conjecture was 
disapproved by Santos \cite{santos12} after a $50$-year 
effort of many experts. Now, some experts, for example 
Santos \cite{santos12a}, believe that the diameter of the convex 
polytope can be bounded by a polynomial of $(m-n)n$. 
This conjectured upper bound for the diameter of the convex 
polytope is much smaller than the best-known quasi-polynomial
upper bounds which are due to Kalai and Kleitman \cite{kalai92}, 
Todd \cite{todd14}, and Sukegawa \cite{sukegawa17}. 
In a recent effort \cite{yang18}, this author showed that
for a given polytope, the diameter is bounded by
$\mathcal{O} \left( \frac{n^3 \Delta}{\det(\A^*)} \right)$,
where $\Delta$ is the largest absolute value among all 
$(n-1) \times (n-1)$ sub-determinants of $\A$ and 
${\det(\A^*)}$ is the smallest absolute value among all 
nonzero $n \times n$ sub-determinants of $\A$.

Finding the diameter of convex polytopes provides only a 
surmised lowest iteration number for which an optimal 
pivot rule may achieve. Finding actually such a pivot 
rule (the way to choose the next neighbor vertex) is also
a difficult problem. Researchers proposed many pivot rules with 
the hope that they may achieve an iteration number 
in the worst case bounded by a polynomial of $m$ and $n$ (see 
\cite{tz93} and references therein). However, since Klee 
and Minty \cite{km72} constructed a cube and showed 
that Dantzig's rule needs an exponential number of iterations 
in the worst case to solve the Klee and Minty cube problem, 
people have showed similar results for almost every popular 
pivot rule \cite{ac78,friedmann11,gs79,jeroslow73,psz08}. 
It is now believed that finding 
a pivot rule that will solve all linear programming 
problems in the worst case in a polynomial time is a very 
difficult problem \cite{smale99}.

Existing pivot rules consider one of many merit criteria 
to select an entering variable. Some popular pivot rules 
are, for example, the most negative 
index in the reduced cost vector (Dantzig's rule), the best
improvement rule, Bland's least index pivoting rule, the
steepest edge simplex rule, Zadeh's rule, among others \cite{tz93}.
Each merit criterion has its own appealing feature. However,
existing simplex algorithms cannot use multiple merits at the 
same time because each of these algorithms 
updates only one variable at a time. In a slightly 
different view, a merit criterion may be a good choice 
in most scenarios but may be a poor choice in some spacial 
case. For example, Dantzig's rule is most efficient for 
general problems \cite{ps14} but performs poorly for the 
Klee-Minty cube \cite{km72}. Therefore, randomized 
pivot rules \cite{ghz98,kp06} that randomly select an entering 
variable from the set of possible entering variables that 
will improve the objective function have been proposed and 
proved to be able to find an optimizer in a polynomial time
on average \cite{kp06}. This shows that using a combination 
of merits in the selection of pivot can be beneficial. 

In this paper, we consider a novel simplex algorithm for linear
programming problem. This algorithm is different from all 
existing simplex algorithms in that it updates two
variables at one iteration. This strategy looks two pivots 
ahead instead of focus only on the next step. We believe 
that this strategy is better than existing pivot rules because 
it looks longer term benefit instead of a short-sighted
one-step achievement. Since the proposed algorithm 
updates two variables at a time, it considers multiple merits 
at the same time in the selection of the next  
vertex in a deterministic way which is different 
from the randomized rules. Numerical test shows that
the proposed algorithm finds the optimal solution in
just one iteration for three variants of Klee-Minty 
problems.  We wish that these features give us some 
hope to find some strong polynomial algorithms to 
solve linear programming problems. We may extend 
the proposed algorithm to select more than two entering 
variables, but there is a trade-off between reducing 
iteration numbers and reducing the cost of each iteration.

After we finished this research, we realized that 
Vitor and Easton \cite{ve18} recently proposed a similar 
idea to update two pivots at a time. Their algorithm
chooses the two entering variables using the most
negative reduced cost criterion rather than using 
a combined merit criteria to select the two entering 
variables. We indicate in Remark \ref{mostNeg} 
that this is not a good strategy. Indeed, 
our numerical test shows that such a choice leads
to an algorithm that needs exponentially many 
iterations to find the solution for Klee-Minty problems
while our proposed algorithm needs just one iteration
for these Klee-Minty problems.

In this paper, we use small letters with bold font for vectors 
and capital letters with bold font for matrices. To save space,
we write the column vector 
$\x =[\x_1^{\Tr}, \x_2^{\Tr}]^{\Tr}$ as $\x=(\x_1, \x_2)$.
The remainder of the paper is organized as follows.
Section 2 describes the proposed algorithm. Section 3
provides two bounds of the iteration numbers of the 
algorithm. Section 4 presents the numerical test results 
for three variants of Klee-Minty cubes and compares the
the performance of the proposed algorithm and Dantzig's
algorithm for randomly generated problems. The 
concluding remarks are in Section 5.

\section{The proposed algorithm}
\label{proposedAlgo}

We consider the primal linear programming problem in
the standard form:
\begin{eqnarray}
\begin{array}{cl}
\min  &  \c^{\Tr}\x, \\ 
\mbox{\rm subject to} 
&  \A\x=\b, \hspace{0.1in} \x \ge \0,
\end{array}
\label{LP}
\end{eqnarray}
where $\A \in {\R}^{m \times n}$, $\b \in {\R}^{m}$, 
$\c \in {\R}^{n}$ are given, and $\x \in {\R}^n$  is the 
vector to be optimized. Associated with the linear 
programming is the dual programming that is also 
presented in the standard form:
\begin{eqnarray}
\begin{array}{cl}
\max   &   \b^{\T}\y, \\
\mbox{\rm subject to}   &  
\A^{\T}\y +\s=\c, \hspace{0.1in} \s \ge \0,
\end{array}
\label{DP}
\end{eqnarray}
where $\y \in {\R}^{m}$ is the dual variable vector, 
and $\s \in {\R}^{n}$ is the dual slack vector. 

A feasible solution of the linear program satisfies
the conditions of $\A \x =\b$ and $\x \ge \0$. 
We will denote by $B \subset \{ 1, 2, \ldots, n \}$ the 
index set with cardinality $| B | = m$ and 
$N = \{ 1, 2, \ldots, n \} \setminus B$ the complementary 
set of $B$ with cardinality $| N |=n-m$ such that 
matrix $\A$ and vector $\x$ can be partitioned as 
$\A=[ \A_B, \A_N]$ and $\x =(\x_B, \x_N)$, moreover the
columns of $\A_B$ are linearly independent and
$\A_B \x_B = \b$, hence $\x_N=\0$. We call this
$\x=(\x_B,\0)$ as the basic feasible solution.
Similarly, we can partition $\c$ and $\s$ according to
the index sets $B$ and $N$ as follows:
\[
\c = \left[ \begin{array}{c}
\c_B \\  \c_N
\end{array} \right], \hspace{0.2in}
\s = \left[ \begin{array}{c}
\s_B \\  \s_N
\end{array} \right]. 
\]
We denote by $\mathcal{B}$ the 
set of all bases $B$. In the discussion below,
we make the following assumptions:
\begin{itemize}
\item[1.] $rank(\A) = m$.
\item[2.] The primal problem (\ref{LP}) has an optimal solution.
\item[3.] Initial basic feasible solution $\x^0$ is given 
and is not an optimizer.
\item[4.] All basic feasible solutions are bounded above and below,
more specifically, for all $i \in B \subset \mathcal{B}$, 
$ \delta \le x_i \le \gamma$.
\end{itemize}
The first three assumptions are standard. The last assumption 
implies that the primal problem (\ref{LP}) is non-degenerate.
Using the $B-N$ partition, we can rewrite the primal problem as
\begin{eqnarray}
\begin{array}{cl}
\min &  \c_B^{\T}\x_B + \c_N^{\T}\x_N,  \\
\mbox{\rm subject to} 
&  \A_B\x_B + \A_N\x_N=\b, 
\hspace{0.1in} \x_B \ge \0, \hspace{0.1in}  \x_N \ge \0.
\end{array}
\label{LP1}
\end{eqnarray}
Since $\A_B$ is non-singular, we can rewrite (\ref{LP1}) as
\begin{eqnarray}
\begin{array}{cl}
\min & \c_B^{\T}\A_B^{-1} \b 
+ (\c_N - \A_N^{\T} \A_B^{-\Tr} \c_B)^{\Tr} \x_N, \\
\mbox{\rm subject to} 
& \x_B = \A_B^{-1} \b - \A_B^{-1}\A_N\x_N, 
\hspace{0.1in} \x_B \ge \0, \hspace{0.1in}  \x_N \ge \0.
\end{array}
\label{LP2}
\end{eqnarray}
Let superscript $k$ represent the $k$th iteration,
the matrices and vectors in the $k$th iteration are then
denoted by $\A_{B^k}$, $\A_{N^k}$, $\c_{B^k}$, $\c_{N^k}$,
$\x_{B^k}$, $\x_{N^k}$, $\s_{B^k}$, and $\s_{N^k}$, where
$\x^k=(\x_{B^k},\x_{N^k})$ is the basic feasible 
solution of (\ref{LP}) with $\x_{B^k} > \0$ and
$\x_{N^k}= \0$. Similarly, we denote by 
$\x^*= (\x_{B^*}, \x_{N^*})$ the 
optimal basic solution of (\ref{LP}) with 
$\x_{B^*}=\A_{B^*}^{-1} \b > \0$ and
$\x_{N^*}=\0$, by $(\y^*, \s^*)$ the optimal basic 
solution of the dual problem (\ref{DP}) with
$\y^*=\A_{B^*}^{-\Tr} \c_{B^*}$, $\s_{B^*}=\0$, and 
$\s_{N^*}=\c_{N^*}-\A_{N^*}^{\Tr}\A_{B^*}^{-\Tr}\c_{B^*} \ge \0$, 
by $z^* = \c^{\Tr} \x^* = \b^{\T}\y^*$ the optimal value.
It is worthwhile to note that the 
partition of $(B^k,N^k)$ keeps updating and it is different 
from the partition $(B^*,N^*)$ before an 
optimizer is found. Let 
\begin{equation}
\bar{\c}_{N^k}^{\Tr} =( \c_{N^k} - \A_{N^k}^{\Tr} 
\A_{B^k}^{-\Tr} \c_{B^k} )^{\Tr}
\label{cNbar}
\end{equation}
be the reduced cost vector. Clearly, if $\bar{\c}_{N^k} \ge \0$, 
an optimizer is found; if $\bar{c}_{j^k} <0$ for some 
$j^k \in N^k$, then an entering variable $x_{\jmath^k}$ 
in the next vertex is chosen from the 
set of $\{ j^k ~|~ \bar{c}_{j^k} <0 \}$ because 
by increasing $x_{\jmath^k}$, the objective function 
$\c^{\Tr} \x=\c_{B^k}^{\T}\x_{B^k}+\bar{c}_{\jmath^k}x_{\jmath^k}$ 
will be reduced. Many different rules have been proposed
for the selection of the entering variable $x_{\jmath^k}$
under the constraint:
\begin{equation}
\jmath^k \in \{ j^k ~|~ \bar{c}_{j^k} <0 \}.
\end{equation}
Once the entering variable is selected, 
existing pivot rules determine the leaving variable using the 
following method: Denote $\bar{\b} = \A_{B^k}^{-1} \b$
and $\bar{\a}_{\jmath^k} = \A_{B^k}^{-1} \A_{\jmath^k}$,
$\jmath^k \in N^k$ is the index of the entering
variable, the leaving variable $x_{\imath}$, 
${\imath} \in B^k$, is determined by the following condition.
\begin{eqnarray}
x_{\imath} = \min_{i \in \{ 1,\ldots,m \}} 
\bar{b}_i / \bar{a}_{\jmath^k,i},   
\hspace{0.1in}  \mbox{subject to}  
\hspace{0.1in} \bar{a}_{\jmath^k,i}>0.
\label{outVar}
\end{eqnarray}
The corresponding step-size is given by
\begin{eqnarray}
\min_{i \in  \{ 1,\ldots,m \}} 
\bar{b}_i / \bar{a}_{\jmath^k,i},  
\hspace{0.1in} \mbox{subject to}
\hspace{0.1in}  \bar{a}_{\jmath^k,i}>0  .
\label{stepsize}
\end{eqnarray}

As we pointed out above, our strategy is to select,
in a deterministic way, two entering variables from 
the set of non-basic variables that will reduce the 
objective function. According to some
extensive computational experience, for example 
\cite{ps14}, Dantzig's rule is the most efficient 
on average among all popular pivot rules (even though 
Dantzig's rule needs exponentially many pivots to find the 
optimal solution for Klee-Minty cubes in the worst case), 
therefore, we select the first entering variable 
$x_{\jmath_1^k}$ using Dantzig's rule: 
\begin{equation}
{\jmath_1^k} :=\{ {\jmath_1^k} ~|~ \bar{c}_{\jmath_1^k} 
= \min_{j^k \in N^k} \bar{\c}_{j^k}  \}.
\label{enteringVar}
\end{equation} 
Kitahara and Mizuno \cite{km13} showed that the 
number of iterations in existing pivot rules is significantly 
affected by the minimum values of all the positive 
elements of primal basic feasible solutions. Carefully
studying Klee-Minty cube and its variants
\cite{greenberg97,km11,ibrahima13} indicates that the other 
entering variable should be determined by taking the
variable among all $j^k \in N^k$ with $\bar{c}_{j^k}<0$
such that a particular $\jmath_2^k$ will maximize the 
step-size defined in (\ref{stepsize}), i.e.,
\begin{equation}
x_{\jmath_2^k} = \max_{\bar{\c}_{j^k}<0}
\Big\{
\min_{i \in  \{ 1,\ldots,m \}} \bar{b}_i / \bar{a}_{j^k,i},
\hspace{0.1in} \mbox{subject to}
\hspace{0.1in}    \bar{a}_{j^k,i}>0
\Big\}.
\label{iniVar2}
\end{equation}
This strategy will be justified again in the proof of Theorem
\ref{thm2} and in the discussion of Remark \ref{longestStep}.
If ${\jmath_2^k}={\jmath_1^k}$ (which means that the most 
negative rule will generate the longest step), then we take the second
entering variable $x_{\jmath_2^k}$ which has the second 
largest step-size.

Now we discuss how to choose the leaving variables. 
To make our notation simple, we drop the 
iteration index $k$ if it does not cause confusion.
Let $\bar{\A}_{(\jmath_1,\jmath_2)} = 
\A_B^{-1} \A_{(\jmath_1,\jmath_2)}$
where $\A_{(\jmath_1,\jmath_2)}$ is composed of the 
$\jmath_1$ and $\jmath_2$ columns of $\A_N$,
and $\bar{\c}_{(\jmath_1,\jmath_2)} < \0$ be the 
two corresponding elements in $\bar{\c}_N$. For the
two entering indexes $(\jmath_1,\jmath_2) \in N^k$ such that that 
$\x_{(\jmath_1,\jmath_2)} =(x_{\jmath_1},x_{\jmath_2}) \ge \0$, 
we need
\begin{equation}
\x_{B^k}^{k+1}=\A_{B^k}^{-1}\b
-\A_{B^k}^{-1}\A_{N^k}\x_{N^k} 
= \bar{\b} - \bar{\A}_{(\jmath_1,\jmath_2)}
\x_{(\jmath_1,\jmath_2)} \ge \0.
\label{updatedXB}
\end{equation}
Therefore, the problem of finding a good new vertex is 
reduced to minimize the following linear programming problem.
\begin{eqnarray}
\begin{array}{cl}
\min  &  \bar{\c}_{(\jmath_1,\jmath_2)}^{\T}
\x_{(\jmath_1,\jmath_2)}, \\ 
\mbox{\rm subject to} 
&  \bar{\A}_{(\jmath_1,\jmath_2)}
\x_{(\jmath_1,\jmath_2)} \le \bar{\b}, 
\hspace{0.1in} \x_{(\jmath_1,\jmath_2)} \ge \0.
\end{array}
\label{twoDim}
\end{eqnarray}
Here the third merit criterion is introduced, which is to
determine the values of the two entering variables to minimize 
the objective function under the constraints of (\ref{twoDim}). 
\begin{lemma}
Let $z^k$ be the value of the objective function of 
(\ref{LP}) at iteration $k$. Then, 
the optimum of the problem (\ref{twoDim})
at iteration $k+1$ gives the minimal $z^{k+1}$
when the most negative and the longest step-size rules
are used for two entering variables, and 
\[
z^{k+1} - z^k = \c^{\Tr} \x^{k+1} - \c^{\Tr} \x^k.
\]
\label{firstLemma}
\end{lemma}
\begin{proof}
From (\ref{updatedXB}), the constraints of 
(\ref{twoDim}) make sure that the updated 
$\x_{B^k}^{k+1} \ge \0$ and the leaving variables are zeros. 
For variables in $N^k$, they will stay in zeros except two 
entering variables $\x_{(\jmath_1,\jmath_2)} \ge \0$, 
and we may write variables in $N^k$ as a block vector 
$\x_{N^k}^{k+1}=(\x_{(\jmath_1,\jmath_2)},\0)$. 
The improvement of $\c^{\Tr} \x^{k+1} - \c^{\Tr} \x^k$
is the optimal solution of (\ref{twoDim}), which is achieved 
when the optimal combination of the two entering variables
is determined.
\hfill \qed
\end{proof}
As problem (\ref{twoDim}) has only two variables, 
the solution is slightly more complicate than the 
selection of a single entering variable in existing pivot 
rules, but is still simple and straightforward. We divide 
$\bar{\A}_{(\jmath_1,\jmath_2)}$ into two parts: 
$\bar{\A}_1$ has the rows with at least one positive element,
and $\bar{\A}_2 \le \0$ has the rows with all elements 
smaller than or equal to zero. Also we partition 
$\bar{\b}=\A_B^{-1} \b$ into the corresponding 
$\bar{\b}_1$ and $\bar{\b}_2$. Since elements in 
$\bar{\A}_2$ are smaller than or equal to zero, in view of 
(\ref{updatedXB}) or (\ref{twoDim}), introducing positive 
entering variables will keep the corresponding
elements in $\x_{B^k}^{k+1}$ to be positive. 
For $\bar{\A}_1$, in view of (\ref{updatedXB})
or (\ref{twoDim}), introducing positive variables may change
the sign of some elements of $\x_{B^k}^{k+1}$. 
If the number of rows in 
$\bar{\A}_1$ is greater than or equal to $2$, for any two 
independent rows $(i_1,i_2)$ of $\bar{\A}_1$
denoted as $\bar{\A}_1(i_1,i_2)$, solving 
\begin{equation}
\bar{\A}_1(i_1,i_2) \left[ \begin{array}{c}
x_{\jmath_1} \\ x_{\jmath_2} \end{array} \right]
= \bar{\b}_1(i_1,i_2)
\label{subEq}
\end{equation}
will give a possible vertex in the convex polygon
defined in (\ref{twoDim}). Therefore 
$\x_2 := (x_{\jmath_1}, x_{\jmath_2}) \ge \0$ 
is a feasible vertex of the polygon if 
$\bar{\A}_1 \x_2 \le \bar{\b}_1$ holds. Otherwise,
it is not feasible and will not be considered further.
Two special feasible vertices, i.e., $\x_2 := (x_{\jmath_1}, 0)$
and $\x_2 :=(0, x_{\jmath_2})$ which correspond to 
the most negative rule and the longest step-size 
rule respectively, should also be considered. 
For all feasible vertices of the convex polygon
defined in (\ref{twoDim}), we select the vertex 
that minimizes the objective function of (\ref{twoDim}).
The corresponding row indexes $(\imath_1,\imath_2)$
that form the selected vertex determine
the leaving variables. If the number of rows in 
$\bar{\A}_1$ is exact one, the longest step pivot rule is used.

The proposed algorithm is therefore as follows:

\begin{algorithm} {\ } \\
Data: Matrix $\A$, vectors $\b$ and $\c$. {\ } \\
Initial basic feasible solution $\x^0$, and its related partitions
$\x_{B^0}$, $\x_{N^0}$, $\A_{B^0}$, $\A_{N^0}$, 
$\c_{B^0}$, $\c_{N^0}$, $(\A_{B^0})^{-1}$, and
$\bar{\c}_{N^0}^{\Tr}=\c_{N^0}^{\Tr}
 - \c_{B^0}^{\Tr} (\A_{B^0})^{-1} \A_{N^0}$.
\begin{itemize}
\item[] While $\min (\bar{\c}_{N^k}) < 0$

\begin{itemize}
\item[] If at least two elements of $\bar{\c}_{N^k}$ are 
smaller than zero

\begin{itemize}
\item[+] The first entering variable $x_{\jmath_1^k}$ is determined 
by Dantzig's rule. For all negative elements of $\bar{\c}_{N^k}$ 
other than the most negative elements $\bar{\c}_{\jmath_1^k}$, 
determine the $x_{\jmath_2^k}$ such that the second entering 
variable will take the longest step. Two special vertices,
$(x_{\jmath_1^k}, 0)$ and $(0, x_{\jmath_2^k})$ are 
obtained.
\item[+] Divide $\bar{\A}_({\jmath_1,\jmath_2})$
into two parts: $\bar{\A}_1$ whose row has positive elements 
and $\bar{\A}_2 \le \0$. Partition $\A_{B^k}^{-1} \b$ into 
the corresponding $\bar{\b}_1$ and $\bar{\b}_2$.
\item[+] If the number of rows of $\bar{\A}_1$ is greater than
or equal to $2$
\begin{itemize}
\item[-] Compute all vertices in two dimensional plane 
by solving (\ref{subEq}).
\item[-] Determine all feasible vertices which satisfy 
$\x_2=(x_{\jmath_1},x_{\jmath_2}) \ge \0$
and $\bar{\A}_1 \x_2 \le \bar{\b}_1$.
\item[-] Find a pair of entering variables among all feasible 
vertices $\x_2$ (including the two special vertices) 
that minimizes the objective 
$[\bar{c}_{\jmath_1}, \bar{c}_{\jmath_2}] \x_2$.
\item[-] Update base $\A_{B^k}$ and $\c_{B^k}$, 
non-base $\A_{N^k}$ and $\c_{N^k}$.
Compute $\A_{B^k}^{-1}$ and 
$\bar{\c}_{N^k}^{\Tr} = \c_{N^k}^{\Tr}
 - \c_{B^k}^{\Tr} \A_{B^k}^{-1} \A_{N^k}$.
\end{itemize}

\item[+] Else if there is only one row in $\bar{\A}_1$
\begin{itemize}
\item[-] The longest step rule is applied.
\item[-] Update base $\A_{B^k}$ and $\c_{B^k}$, 
non-base $\A_{N^k}$ and $\c_{N^k}$.
Compute $\A_{B^k}^{-1}$ and 
$\bar{\c}_{N^k}^{\Tr} = \c_{N^k}^{\Tr} 
- \c_{B^k}^{\Tr} \A_{B^k}^{-1} \A_{N^k}$.
\end{itemize}
\item[+] end  (if)


\end{itemize}
\item[] Else if only one element of $(\bar{\c}_{N^k})$ is negative, 
\begin{itemize}
\item[+] Dantzig's rule (which is also the longest rule) is applied.
\item[+] Update base $\A_{B^k}$ and $\c_{B^k}$, 
non-base $\A_{N^k}$ and $\c_{N^k}$.
Compute $\A_{B^k}^{-1}$ and 
$\bar{\c}_{N^k}^{\Tr} = \c_{N^k}^{\Tr}
 - \c_{B^k}^{\Tr} \A_{B^k}^{-1} \A_{N^k}$.
\end{itemize}
\item[] end (if)

\item[] $k \Leftarrow k+1$.

\end{itemize}
\item[] end (while)

\end{itemize}
\label{mainAlg}
\end{algorithm}

\begin{remark}
We can modify the algorithm by selecting two entering
variables using the indexes corresponding to the {\bf two} most 
negative elements in $\bar{\c}_{N^k}$ as \cite{ve18}. In the 
numerical test section, we will see that this is a not
a good strategy.
\label{mostNeg}
\end{remark}

\section{Bounds of the iteration numbers of the algorithm}

In this section, we provide two upper bounds of the iteration
numbers for the proposed algorithm using the strategy 
developed in \cite{km11,km13,ye11}. 

Let $r$ be any real number and $\lceil r \rceil$ be the 
smallest integer bigger than $r$. Let $\gamma_P^*$
be the maximum value of all elements of $\x^*$
and 
\begin{equation}
\gamma_D = \max_{k} \{ \gamma_D^k \} =
\max_{k} \Big\{ \max_{j^k \in N^k} 
\{ -\bar{c}_{j^k} ~|~ \bar{c}_{j^k} < 0 \} \Big\}.
\label{gammaD}
\end{equation}
Let $(B^k,N^k)$ be the partitions of base and 
non-base variables at iteration $k$ and $(B^*,N^*)$
be the partitions of base and non-base variables of
the optimization solution. Let $\x^*$ be partitioned 
using $(B^k,N^k)$ but not $(B^*,N^*)$, i.e.,
\[
\x^*= \left[ \begin{array}{c}
\x_{B^k}^* \\ \x_{N^k}^* \end{array} \right].
\]
The first lemma is derived using exactly the same 
argument but states a slightly improved result 
of \cite{km11,km13}.

\begin{lemma} (Kitahara and Mizuno)
Let $\x^*$ be partitioned using $(B^k,N^k)$
and $z^*$ be the optimal value of (\ref{LP}), we have
\begin{equation}
\c^{\Tr} \x^k - z^* 
\le \gamma_D^k \| \x_{N^k}^* \|_1.
\label{1stEst}
\end{equation}
\label{firstL}
\end{lemma}
\begin{proof}
Since $\x^*$ is partitioned using $(B^k,N^k)$, 
we have $(\x_{B^k}^* , \x_{N^k}^*) \ge \0$, and
\begin{eqnarray} 
\A_{B^k}\x_{B^k}^* +\A_{N^k}\x_{N^k}^*= \b. \nonumber
\end{eqnarray} 
This gives
\begin{eqnarray} 
\x_{B^k}^*  
= \A_{B^k}^{-1} \b -\A_{B^k}^{-1} \A_{N^k}\x_{N^k}^*.
\nonumber 
\end{eqnarray}  
Therefore, we have
\begin{eqnarray} 
\c^{\Tr} \x^* & = & 
\c_{B^k}^{\Tr}\x_{B^k}^* + {\c}_{N^k}^{\Tr}\x_{N^k}^*  
\nonumber \\
& = & \c_{B^k}^{\Tr}  \A_{B^k}^{-1} \b 
-{\c}_{B^k}^{\Tr}\A_{B^k}^{-1}\A_{N^k}\x_{N^k}^*
+{\c}_{N^k}^{\Tr}\x_{N^k}^*   
\nonumber \\
 & = & \c_{B^k}^{\Tr}  \A_{B^k}^{-1} \b 
+ ({\c}_{N^k}^{\Tr}-{\c}_{B^k}^{\Tr}\A_{B^k}^{-1}\A_{N^k})
\x_{N^k}^*.
\end{eqnarray}
Using this relation and (\ref{cNbar}), we have 
\begin{eqnarray}
z^* & = & \c^{\Tr} \x^*   \nonumber \\
& = & \c_{B^k}^{\Tr}  \A_{B^k}^{-1} \b 
+ \bar{\c}_{N^k}^{\Tr}\x_{N^k}^*  \nonumber \\ 
& \ge &  \c^{\Tr} \x^k -  \gamma_D^k \| \x_{N^k}^* \|_1.
\label{est1}
\end{eqnarray}
This finishes the proof.
\hfill\qed
\end{proof}

\begin{remark}
If $B^* \ne B^k$, i.e., $\x^k$ is not an optimizer, from
(\ref{est1}), it must have $\bar{\c}_{N^k}^{\Tr}\x_{N^k}^* <0$.
Therefore, there is a $j^k \in N^k$ such that 
\begin{equation}
\bar{c}_{j^k} <0 \hspace{0.1in} \mbox{and}
\hspace{0.1in} x_{j^k}^*>0.
\label{NkBstar}
\end{equation}
This means that for $j^k \in N^k \cap B^*$,
$x_{j^k}$ should be the entering variable.
The problem is that one does not know $B^*$
before an optimizer is found.
\label{totalD}
\end{remark}

We may also partition $\x^k$ using $(B^*,N^*)$ as
\[
\x^k= \left[ \begin{array}{c}
\x_{B^*}^k \\ \x_{N^*}^k \end{array} \right].
\]
This gives
\begin{eqnarray} 
\A_{B^*}\x_{B^*}^k +\A_{N^*}\x_{N^*}^k= \b, \nonumber
\end{eqnarray} 
and
\begin{eqnarray} 
\x_{B^*}^k  
= \A_{B^*}^{-1} \b -\A_{B^*}^{-1} \A_{N^*}\x_{N^*}^k.
\nonumber 
\end{eqnarray}  
Similar to the derivation of (\ref{1stEst}), we have
\begin{eqnarray} 
\c^{\Tr} \x^k  = \c^{\Tr} \x^* + 
({\c}_{N^*}- {\c}_{B^*}\A_{B^*}^{-1} \A_{N^*})^{\Tr} \x_{N^*}^k
= z^* + \bar{\c}_{N^*}^{\Tr}\x_{N^*}^k.
\nonumber 
\end{eqnarray} 
If $x_{j}^k > 0$, we have $x_{j}^k \in B^k$. For 
$j \in {N^*} \cap B^k$, since $\bar{\c}_{N^*} \ge \0$, we have
\begin{eqnarray} 
z^* & \ge  \c^{\Tr} \x^k -   
\max \{ x_j^k   ~|~ j  \in N^* \cap B^k \}
\| \bar{\c}_{N^*}  \|_1.
\end{eqnarray}

\begin{remark}
If $N^* \ne N^k$, i.e., $\x^k$ is not an optimizer, it
must have $\bar{\c}_{N^*}^{\Tr}\x_{N^*}^k >0$.
Therefore, there is a $j^* \in N^*$ such that 
\begin{equation}
\bar{c}_{j^*} >0 \hspace{0.1in} \mbox{and}
\hspace{0.1in} x_{j^*}^k>0.
\label{NkBstar2}
\end{equation}
This means that for $j^* \in N^* \cap B^k$,
$0 < x_{j^*}^k \in B^k$ should be the leaving variable.
The problem is that one does not know $N^*$
before an optimizer is found.
\label{totalDiff}
\end{remark}

Let $\gamma_{\ell}=\min_{k} x_{\jmath_2^k}$, where 
$x_{\jmath_2^k}$ is defined in (\ref{iniVar2}), i.e., 
$x_{\jmath_2^k}$ is the longest step among all
possible entering variables with $\bar{c}_{j^k} <0$
and $j^k \in N^k$; and define
\begin{equation}
\delta_D =\min_{k} \delta_D^k
= \min_{k}  \Big\{ \min \{ -\bar{c}_{j^k} ~| ~j^k \in N^k
\hspace{0.1in}\mbox{and} \hspace{0.1in} \bar{c}_{j^k}<0 \}
\Big\}.
\label{deltaD}
\end{equation}
Considering Algorithm \ref{mainAlg}, our next lemma
is an improvement of the one in \cite{km13}.

\begin{lemma}
Let $\x^k$ and $\x^{k+1}$ be the $k$th and $(k+1)$th 
iterates generated by Algorithm \ref{mainAlg}. 
If $\x^k$ is not optimal and $\x^k \neq \x^{k+1}$, 
then, we have
\begin{equation}
\c^{\Tr} \x^k - \c^{\Tr} \x^{k+1} \ge \delta_D \gamma_{\ell}.
\label{myIneq}
\end{equation}
\label{diffObj}
\end{lemma}
\begin{proof}
Since $\x^k \neq \x^{k+1}$, from Lemma
\ref{firstLemma}, the difference of the objective 
functions between $k$th and $(k+1)$th iterations
is actually the solution of (\ref{twoDim}), which is
smaller than the special case when only one entering
variable $x_{\jmath_2^k}$, which would generate 
the longest step among $\bar{c}_{j^k}<0$ for all 
$j^k \in N^k$, is selected. Let 
$\bar{\x}_{(\jmath_1^k,\jmath_2^k)}$ be the optimal
solution of (\ref{twoDim}) at iteration $k$. Therefore
\begin{eqnarray}
\c^{\Tr} \x^k - \c^{\Tr} \x^{k+1} 
& = & - \bar{\c}_{(\jmath_1^k,\jmath_2^k)}^{\T}
\bar{\x}_{(\jmath_1^k,\jmath_2^k)}  
\nonumber \\
& \ge &  -\bar{c}_{\jmath_2^k}~{x}_{\jmath_2^k}
\nonumber \\
& \ge &  \delta_D \gamma_{\ell}.
\label{myIneq1}
\end{eqnarray}
This finishes the proof. 
\hfill \qed
\end{proof}

\begin{remark}
Lemma \ref{diffObj} says that for Algorithm \ref{mainAlg},
the objective value decreases in every iteration by at least a 
constant $\delta_D \gamma_{\ell}$.
\label{constantD}
\end{remark}

From Lemmas \ref{firstL}, \ref{diffObj} and 
Remark \ref{constantD}, it is easy to show that the 
following upper bound of iteration numbers
of Algorithm \ref{mainAlg} holds.

\begin{theorem}
Suppose that we generate a sequence of basic feasible
solutions by Algorithm \ref{mainAlg} from an initial
iterate $\x^0$. Then, the number of total iterations is
bounded above by 
\begin{equation}
\Bigl\lceil 
\frac{\c^{\Tr} \x^0 - z^*}{\delta_D \gamma_{\ell}}
\Bigr\rceil \le
\Bigl\lceil
\frac{\gamma_D^0\| \x^* \|_1 }
{\delta_D \gamma_{\ell}}
\Bigr\rceil
\label{bound}
\end{equation}
\label{mainThe}
\end{theorem}
\begin{proof}
Since every iteration will reduce the objective function by at least a
constant $\delta_D \gamma_{\ell}$, and the total difference 
between the initial objective function and the optimal 
objective function is $\c^{\Tr} \x^0 - z^*$,  we need at most 
\[
\Bigl\lceil
\frac{\c^{\Tr} \x^0 - z^* }
{\delta_D \gamma_{\ell}}
\Bigr\rceil
\]
iterations to find the optimal solution.
The bound of the left side of (\ref{bound}) is obtained.
By the definition of $\gamma_D^k$, we have
\[
\bar{\c}_{N^k}^{\Tr}\x_{N^k}^* \ge 
-{\gamma_D^k \| \x_{N^k}^* \|_1 }.
\]
Therefore, for initial step, the last inequality of 
(\ref{est1}) can be replaced by
\[
\c^{\Tr} \x^0 - z^* \le \gamma_D^0 \| \x^* \|_1 .
\]
This shows the inequality of (\ref{bound}).
\hfill\qed
\end{proof}

\begin{remark}
The upper bound given in Theorem \ref{mainThe} is smaller than
the one in \cite{km13} because (a) $\gamma_{\ell}$ is the smallest
value in all longest steps among all iterates while the 
corresponding number in \cite{km13} is the smallest value 
in all nonzero components among all iterates $\x^k$, 
(b) $\| \x^* \|_1 $ depends only on the optimal solution
of $\x^*$, and (c) $\gamma_D^0$ depends only on the 
vector $\c$.
\end{remark}

Now, we present an upper bound in terms of only $\delta$
and $\gamma$ defined in Assumption 4.

\begin{theorem}
Assume that the $k$th iterate generated by Algorithm \ref{mainAlg}
is not an optimizer. Let
\begin{equation}
t = 
m \frac{\gamma}{\delta}
\log\left( m \frac{\gamma}{\delta} \right)
\label{tIter}
\end{equation}
then there is a $\bar{j} \in B^k$, a corresponding
$x_{\bar{j}}^k>0$, after at most another $\lceil t \rceil$ iterations, 
$x_{\bar{j}}^{k+t}$ becomes zero and stays there since then.
\label{thm2}
\end{theorem}
\begin{proof}
In view of (\ref{est1}) in Lemma \ref{firstL}, since 
$\x^*$ has at most $m$ nonzero elements and
$\bar{\c}_{N^k}^{\Tr}\x_{N^k}^* \ge -\gamma_D^k (m \gamma)$,
we have 
\[
{\c}^{\Tr} \x^k -z^* \le m \gamma \gamma_D^k.
\]
Using this inequality, together with Lemma 
\ref{firstLemma}, (\ref{myIneq1}) in 
Lemma \ref{diffObj} and (\ref{enteringVar}), we have
\begin{eqnarray}
\c^{\Tr} \x^k - \c^{\Tr} \x^{k+1} 
& = & - \bar{\c}_{(\jmath_1^k,\jmath_2^k)}^{\T}
\bar{\x}_{(\jmath_1^k,\jmath_2^k)}  
\nonumber \\
& \ge &  -\bar{c}_{\jmath_1^k}~{x}_{\jmath_1^k}
\nonumber \\
& \ge &  \gamma_D^k \delta
\nonumber \\
& \ge &  \frac{\delta}{m\gamma}
\left( {\c}^{\Tr} \x^k -z^*  \right).
\nonumber 
\end{eqnarray}
This shows
\begin{eqnarray}
\c^{\Tr} \x^k - z^* - (\c^{\Tr} \x^{k+1} -z^*)
\ge  \frac{\delta}{m\gamma}
\left( {\c}^{\Tr} \x^k -z^*  \right)
\nonumber 
\end{eqnarray}
or equivalently
\begin{eqnarray}
\c^{\Tr} \x^{k+1} -z^*
\le  \left( 1 - \frac{\delta}{m\gamma} \right)
\left( {\c}^{\Tr} \x^k -z^*  \right).
\nonumber 
\end{eqnarray}
Therefore, for any integer $t>0$, we have
\begin{eqnarray}
\frac{\c^{\Tr} \x^{k+t} -z^*} {\c^{\Tr} \x^k -z^*}
\le  \left( 1 - \frac{\delta}{m\gamma} \right)^t.
\label{myIneq2}
\end{eqnarray}
Since $| B^k |=m$ and 
\[
\c^{\Tr} \x^k -z^* = \x^{k^{\Tr}} \s^* 
= \sum_{j \in B^k} x_j^k s_j^*,
\]
there must have a $\bar{j} \in B^k$ such that 
\[
x_{\bar{j}}^k s_{\bar{j}}^* \ge \frac{1}{m} (\c^{\Tr} \x^k -z^* ).
\]
Using Assumption 4, $\gamma \ge x_{\bar{j}}^k > 0$, we have
\begin{eqnarray}
s_{\bar{j}}^* \ge \frac{1}{mx_{\bar{j}}^k } (\c^{\Tr} \x^k -z^* )
\ge \frac{1}{m \gamma } (\c^{\Tr} \x^k -z^* ).
\label{myIneq3}
\end{eqnarray}
Moreover, for any integer $t >0$, we have
\[
\c^{\Tr} \x^{k+t} -z^* = \s^{*^{\Tr}}\x^{k+t} 
\ge x_{\bar{j}}^{k+t} s_{\bar{j}}^*,
\]
this gives
\begin{eqnarray}
x_{\bar{j}}^{k+t} \le  \frac{\c^{\Tr} \x^{k+t} -z^* }{s_{\bar{j}}^*}.
\label{myIneq4}
\end{eqnarray}
Substituting (\ref{myIneq3}) and (\ref{myIneq2}) into 
(\ref{myIneq4}) gives
\begin{eqnarray}
x_{\bar{j}}^{k+t} 
\le m \gamma \frac{\c^{\Tr} \x^{k+t} -z^* }
{\c^{\Tr} \x^k -z^* }
\le m \gamma \left( 1 - \frac{\delta}{m\gamma} \right)^t.
\label{myIneq5}
\end{eqnarray}
Substituting (\ref{tIter}) into (\ref{myIneq5}) and 
using the identity $x^{\log_b y}= y^{\log_b x}$ and the
inequality $\log(1-x) \le -x$ for all $x \le  1$, we have
\begin{eqnarray}
x_{\bar{j}}^{k+t} & \le &  m \gamma \left( 1 - \frac{\delta}{m\gamma} \right)^{m \frac{\gamma}{\delta}
\log\left( m \frac{\gamma}{\delta} \right)}
\nonumber \\
& = & m \gamma \left[ \left( 1 - \frac{\delta}{m\gamma} 
\right)^{\log\left( m \frac{\gamma}{\delta} \right)}
\right]^{m \frac{\gamma}{\delta}}
\nonumber \\
& = & m \gamma \left[ 
\left( m \frac{\gamma}{\delta} \right)^{\log
\left( 1 - \frac{\delta}{m\gamma} \right)}
\right]^{m \frac{\gamma}{\delta}}
\nonumber \\
& \le & m \gamma \left[ \left( 
m \frac{\gamma}{\delta} \right)^{- \frac{\delta}{m\gamma}} 
\right]^{m \frac{\gamma}{\delta}} \le \delta.
\label{myIneq6}
\end{eqnarray}
Therefore, after at most $\lceil t \rceil$ iterations, 
$x_{\bar{j}}^{k+t}< \delta$ holds. In view of Assumption 4,
we conclude that 
$x_{\bar{j}}^{k+t}$ is not a basic variable of $B^{k+t}$
and (\ref{myIneq2}) asserts that it will not be a basic variable
thereafter.
\hfill\qed
\end{proof}
The scenario described in the theorem can occur at most
one time for each optimal non-basic variable and since there are  
$n-m$ non-basic optimal variables, we have the following theorem.

\begin{theorem}
For the double-pivot algorithm \ref{mainAlg}, it needs at most
$(n-m) \Bigl\lceil
m \frac{\gamma}{\delta}
\log\left( m \frac{\gamma}{\delta} \right)
\Bigr\rceil$ iterations to find the optimal solution of (\ref{LP}).
\label{main2}
\end{theorem}

\begin{remark}
The way of selecting $x_{\bar{j}}^{k}$ below (\ref{myIneq2})
implies that one should consider the entering variable that takes 
the longest step because this entering variable has a better
chance to replace an optimal non-basic variable.
\label{longestStep}
\end{remark}

It seems that both $\gamma$ and $\delta$ in
Theorem \ref{main2} are very difficult to obtain, 
and the significance of the upper bound
is questionable. As a matter of fact,
using the identical argument in \cite{km13a}, 
we can apply this bound to some special linear programming
problems, such as LP with a totally unimodular matrix and 
Markov Decision Problem with a fixed discount rate, and 
show that this bound can be related to only the
problem sizes $m$, $n$, and $\| \b \|_1$, therefore,
the double-pivot algorithm solves these special 
LP problems in a strongly polynomial time.

For LP whose matrix $\A$ is totally unimodular and all the 
element of $\b$ are integers, all basic feasible solutions are
integers, which means that $\delta \ge 1$. Notice that 
all elements of $\A_B^{-1}$ are $\pm 1$ or $0$, we have
$\gamma \le \| \b \|_1$. A corollary of Theorem \ref{main2}
is as follows:
\begin{corollary}
For LP whose matrix $\A$ is totally unimodular and all the 
element of $\b$ are integers, the double-pivot Algorithm 
\ref{mainAlg} needs at most
$(n-m) \Bigl\lceil m \| \b \|_1
\log\left( m \| \b \|_1 \right) \Bigr\rceil$ 
iterations to find the optimal solution of the linear programming
problem. If $\b$ is also totally unimodular, the double-pivot Algorithm 
\ref{mainAlg} needs at most
$(n-m) \Bigl\lceil m^2
\log\left( m^2 \right) \Bigr\rceil$ 
iterations to find the optimal solution of the linear programming
problem. 
\end{corollary}

For Markov Decision Problem with a fixed discount rate, Ye 
\cite{ye11} showed (1) $\delta \ge 1$, and (2) for the 
constant discount rate $\theta <1$, 
$\gamma \le \frac{m}{1-\theta}$. Therefore, the second
corollary of Theorem \ref{main2} is as follows:
\begin{corollary}
For Markov Decision Problem with a fixed discount rate, 
the double pivot Algorithm \ref{mainAlg} needs at most
$(n-m) \Bigl\lceil \frac{m^2}{1-\theta}
\log\left( \frac{m^2}{1-\theta} \right) \Bigr\rceil$ 
iterations to find the optimal solution of the linear programming
problem.
\end{corollary}

The tightness of the two bounds in Theorems 
\ref{mainThe} and \ref{main2} can be
seen from the following problem provided in \cite{km11}:
\begin{eqnarray}
\begin{array}{cl}
\min & -\sum_{i=1}^m   x_i \\
\mbox{subject to} & 
x_1 + x_{m+1} = 1, \\
& 2 \sum_{i=1}^{k-1} x_i +x_k + x_{m+k} = 2^k-1
 \hspace{0.1in} k=2, \ldots, m, 
  \\
& x_i \ge 0 \hspace{0.1in} i=1, \ldots, 2m.
\end{array}
\label{3rdStandard}
\end{eqnarray}
Assuming that the initial point is taken as 
$\x^0 =[0,\ldots,0,1,\ldots,1]$ (there are $m$ zeros
and $n-m=m$ ones)
and Dantzig's rule is used, for this problem, Kitahara 
and Mizuno showed \cite{km11} that the bound of 
Theorem \ref{main2} is reduced to $\lceil (2m \log 2)2^m \rceil$,
while the actual iteration number is $2^m-1$. The estimated
bound is reasonably tight. We show that the bound of 
Theorem \ref{mainThe} is much tighter than the one
of Theorem \ref{main2}. For this problem, it is easy
to see that the first $m$ variables of the optimal solution are
$[x_1^*,\ldots,x_m^*]=[ 0, \ldots, 0,2^m-1]$ with 
optimal objective function $-(2^m-1)$ and the objective 
function at initial $\x^0$ is zero. Therefore, we have
$\c^{\Tr}\x^0-z^*=2^m-1$.  Since $\c_{B^0}=\0$ 
and $\c_{N^0}=(-1, \ldots, -1)=\bar{\c}_{N^0}$,
this shows that $\delta_D^0 = 1$ (see (\ref{deltaD})). 
In the first iteration, noticing that
$B^0=\{ m+1, m+2,\ldots, 2m \}$ and the entering variable 
$\bar{x}_{\jmath_2^k}=x_m=2^m-1$, i.e., $\gamma_{\ell}=2^m-1$. This shows that the upper 
bound of Theorem \ref{mainThe} is reduced to 
$\bigl\lceil(2^m-1)/\gamma_{\ell}\bigr\rceil$,
i.e., it needs only one iteration to find the optimal 
solution. This claim is also verified in the numerical 
test in the next section for several variants of 
Klee-Minty cube.

\section{Numerical test}

Numerical tests for the proposed algorithm have been done
for two purposes. First, we would like to verify that the algorithm
indeed solves Klee-Minty cube problems efficiently. Second,
we would like to know if this algorithm is competitive to 
the Dantzig's pivot rule for randomly generated LP problems
as we known that Dantzig's rule is the most efficient
deterministic pivot rule for general problems \cite{ps14}.

\subsection{Test on Klee-Minty cube problems}

Klee-Minty cube and its variants have been used to prove that
several popular simplex algorithms need exponential 
number of iterations in the worst case to find an optimizer. 
In this section, three variants of Klee-Minty cube 
\cite{greenberg97,ibrahima13,km11} are used to test
the proposed algorithm.

The first variant of Klee-Minty cube is given in \cite{greenberg97}:
\begin{eqnarray}
\begin{array}{cl}
\min & -\sum_{i=1}^m 2^{m-i} x_i \\
\mbox{subject to} & 
\left[ 
\begin{array}{cccccc}
1 & 0 & 0 & \ldots & 0 & 0 \\
2^2 & 1 & 0 & \ldots & 0 & 0 \\
2^3 & 2^2  & 1 & \ldots & 0 & 0 \\
\vdots &  \vdots & \vdots &  \ddots & 0 & 0 \\
2^{m-1} &  2^{m-2} & 2^{m-3} & \ldots  & 1 & 0 \\
2^m &  2^{m-1} & 2^{m-2} & \ldots &  2^2  & 1 \\
\end{array}
\right]
\left[ \begin{array}{c}
x_1 \\ x_2 \\ \vdots \\  \vdots \\ x_{m-1} \\ x_m
\end{array} \right] 
\le 
\left[ \begin{array}{c}
5 \\ 25 \\ \vdots \\  \vdots \\ 5^{m-1} \\ 5^m
\end{array} \right] 
  \\
& x_i \ge 0 \hspace{0.1in} i=1, \ldots, m.
\end{array}
\label{1stProblem}
\end{eqnarray}
The optimizer is $[ 0, \ldots, 0,5^m ]$ with optimal 
objective function $-5^m$.

The second variant of Klee-Minty cube is given in \cite{ibrahima13}:
\begin{eqnarray}
\begin{array}{cl}
\min & -\sum_{i=1}^m 10^{m-i} x_i \\
\mbox{subject to} & 
2 \sum_{j=1}^{i-1} 10^{i-j} x_j +x_i \le 100^{i-1}
 \hspace{0.1in} i=1, \ldots, m,  \\
& x_i \ge 0 \hspace{0.1in} i=1, \ldots, m.
\end{array}
\label{2ndProblem}
\end{eqnarray}
The optimizer is $[ 0, \ldots, 0,10^{2(m-1)}]$ with optimal 
objective function $-10^{2(m-1)}$.

The third variant of Klee-Minty cube is given in \cite{km11} 
(its standard form was discussed in the previous section):
\begin{eqnarray}
\begin{array}{cl}
\min & -\sum_{i=1}^m   x_i \\
\mbox{subject to} & 
x_1 \le 1, \\
& 2 \sum_{i=1}^{k-1} x_i +x_k \le 2^k-1
 \hspace{0.1in} k=2, \ldots, m, 
  \\
& x_i \ge 0 \hspace{0.1in} i=1, \ldots, m.
\end{array}
\label{3rdProblem}
\end{eqnarray}
The optimizer is $[ 0, \ldots, 0,2^m-1]$ with optimal 
objective function $-(2^m-1)$.

The test results are summarized in Table \ref{table1}.
All initial points are selected as $[0, \ldots, 0]^{\Tr}$
from which all popular pivot algorithms needs $2^m-1$
iterations to find the optimal solution.
For the first variant of Klee-Minty cube \cite{greenberg97}, 
using the most two negative elements of $\bar{c}_{N^k}$ 
to choose the entering variables (the strategy used in
\cite{ve18} as described in Remark 
\ref{mostNeg}) is better than the strategy of Dantzig's rule
which uses the most negative element 
of $\bar{c}_{N^k}$ to choose the entering variable. 
The pivot rule with the most two negative elements 
uses half of the iterations of Dantzig's rule but the iteration
numbers still increase exponentially fast. When the size
$m \ge 18$, the program freezes because iteration numbers
are very big and the computational time is very long.
Algorithm \ref{mainAlg} is much more impressive. For
all problems in three variants, only one iteration is needed 
to find the optimal solution, except for the problem with dimension
$m=200$ in variant 2 \cite{ibrahima13} because Matlab R2016a
on computer Dell Inspiron 3847 cannot store the big value (bigger 
than 10E+310) in vector $\b$. This verifies that the estimated
bound of Theorem \ref{mainThe} is attainable.

We also compared the tests result with the one in 
\cite{ibrahima13} which uses randomized pivot method.
For $m=100$, the randomized pivot method uses more
than $1000$ iterations to find the solution for a variant 
of Klee-Minty cube on average of $200$ runs; 
for $m=200$, the randomized pivot method uses  
more than $5000$ iterations to find the solution on 
average of $200$ runs. Using Algorithm \ref{mainAlg},
it takes one iteration for these problems. 
The proposed double-pivot algorithm is much
more efficient than the randomized algorithm
for these Klee-Minty cube problems. This result
justifies a moderate computational cost increase
in each iteration.

\begin{longtable}{|c|c|c|c|c|c|}
\hline          
\multirow{2}{*}{Problem} & \multicolumn{3}{|c|}
{Klee-Minty Variant 1 \cite{greenberg97}}  
& \multicolumn{1}{|c|}{Variant 2 \cite{ibrahima13}}   
& \multicolumn{1}{|c|} {Variant 3 \cite{km11}}  \\ \cline{2-6}
{size}  & {Dantzig} & Remark \ref{mostNeg} & Alg. 2.1 &  
Alg. 2.1  & Alg. 2.1    \\ 
\hline
2     & 3            & 2             &   1   & 1  &  1    \\ \hline
3     & 7            &  4            &   1   & 1  &  1    \\ \hline
4     &  15          & 8            &   1   & 1  &  1    \\ \hline
5     &  31          & 16          &  1    & 1 &  1     \\ \hline
6     &  63          & 32          &  1    & 1 &  1      \\ \hline
7     &  127         & 64          &  1   & 1 &  1     \\ \hline   
8     &  255         & 128        &  1   & 1 & 1      \\ \hline
9     &  511         & 256        &  1   & 1 & 1     \\ \hline
10   & 1023         & 512        &  1   & 1 & 1     \\ \hline
11 & $2^{11}-1$ & 1024       &  1   & 1 & 1      \\ \hline
12 & $2^{12}-1$ & $2^{11}$ &  1  & 1 & 1       \\ \hline
13 & $2^{13}-1$ & $2^{12}$ &  1  &  1 & 1    \\ \hline
14 & $2^{14}-1$ & $2^{13}$ &  1  & 1 & 1       \\ \hline
15 & $2^{15}-1$ & $2^{14}$ &  1  &  1 &  1        \\ \hline
16 & $2^{16}-1$ & $2^{15}$ &  1  & 1 &  1       \\ \hline
17   &  -             & $2^{16}$ &  1  & 1 &  1       \\ \hline
18   &  -             & -             &  1   & 1  &  1      \\ \hline
19   & -              & -             & 1  &  1 &  1        \\ \hline
20   &  -             & -             & 1  &  1 &  1     \\ \hline
21   &  -             & -             & 1  &  1 &  1      \\ \hline
22   &  -             & -             & 1  &  1 & 1       \\ \hline
23   &  -             & -             & 1  &  1 & 1          \\ \hline
24   & -              & -             & 1  &  1 & 1        \\ \hline
25   &  -             & -             & 1  &  1 & 1        \\ \hline
26   &  -             & -             & 1  &  1 & 1        \\ \hline
27   &  -             & -             & 1  &  1 & 1       \\ \hline
28   &  -             & -             & 1  & 1 & 1        \\ \hline
29   &  -             & -             & 1  &  1 & 1        \\ \hline
30   &  -             & -             & 1  & 1 &  1     \\ \hline
100  & -             & -           & 1 & 1 &  1       \\ \hline
200  & -             & -           & 1 & - &  1      \\ \hline
\caption{Iteration count for three Klee-Minty variants}
\label{table1}
\end{longtable}

\subsection{Test on randomly generated problems}

We also tested and compared Algorithm~\ref{mainAlg}
and Dantzig's pivot algorithm using randomly generated
problems. Some details of the implementation of 
Algorithm~\ref{mainAlg} are provided here for readers 
who are interested in repeating the test. 

Note that the burden of the algorithm is to repeatedly
solve the two dimensional linear programming problem
(\ref{twoDim}). To reduce the computational cost, we 
partitioned  
\[
\bar{\A}_{(\jmath_1,\jmath_2)}=
\left[ \begin{array}{c} \bar{\A}_1 \\ \bar{\A}_2
\end{array} \right], \hspace{0.1in}
\bar{\b}=\left[ \begin{array}{c} 
\bar{\b}_1 \\ \bar{\b}_2
\end{array} \right], 
\]
and showed in Section \ref{proposedAlgo} that we 
only need to consider a subset of the constraints
$\bar{\A}_1 \x_2 \le \bar{\b}_1$ in (\ref{twoDim}).
For large problems, there are still many redundant 
constraints which can easily be removed. For any
constraint in $\bar{\A}_1 \x_2 \le \bar{\b}_1$,
since $\bar{b}_i > 0$, these constraints 
can be rewritten as $A_{i1} x_1 +A_{i2} x_2 \le 1$
by dividing each row by $\bar{b}_i$. 
Therefore, we can further divide these constraints 
into five categories so that we can use the following
heuristics to remove more redundant constraints. 
We use Matlab notations which make it easy to describe
the process.

{\it Category 1: $A_{i1}>0$ and $A_{i2}>0$}

For constraints in this category, we remove the
redundant constraints as follows (see Figure \ref{case1}):

\begin{figure}[htb]
\centerline{\includegraphics[height=6cm,width=6.5cm]
{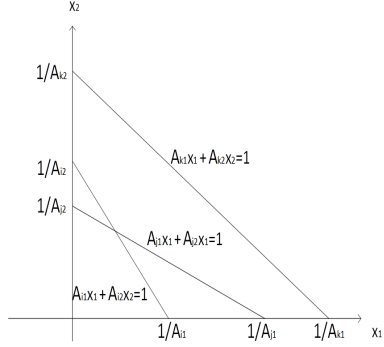}}
\caption{Constraints in Category 1}
\label{case1}
\end{figure}

\begin{itemize}
\item[] Let $L_1$ be the number of constraints in this
category. We find the smallest intercept in x-axis 
$[ix,{i}]=\min_i \left\{ \frac{1}{A_{i1}} \right\}$
and the smallest intercept in y-axis 
$[iy,{j}]=\min_i \left\{ \frac{1}{A_{i2}} \right\}$.
\item[] If ${i}={j}$, all constraints except 
${i}$th constraint are redundant. Denote the candidate 
non-redundant constraint set $C_1= \{ {i} \}$.
\item[] If ${i} \neq {j}$, solving the linear system 
composed of ${i}$th and ${j}$th equations gives
$(x_1, x_2)$. Set the candidate non-redundant constraint 
set $C_1=\{ {i},{j} \}$.
\item[] For $k=1:L_1$
\begin{itemize}
\item[] If $A_{k1}x_1+A_{k2}x_2>1$, add index $k$ into
candidate non-redundant constraint set $C_1$. Otherwise, the $k$th
constraint is redundant.
\end{itemize}
\item[] End (For)
\end{itemize}


{\it Category 2: $A_{i1}>0$ and $A_{i2}<0$}

For constraints in this category, the non-redundant
constraints are selected as follows (see Figure \ref{case2}):
\begin{figure}[htb]
\centerline{\includegraphics[height=6cm,width=6.5cm]
{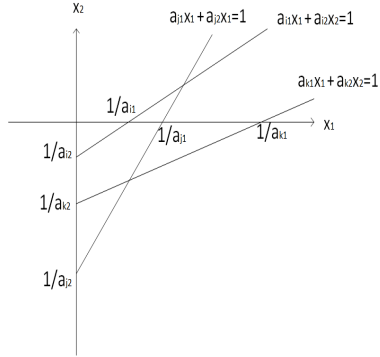}}
\caption{Constraints in Category 2}
\label{case2}
\end{figure}

\begin{itemize}
\item[] Let $L_2$ be the number of constraints in this
category.
\item[] If $L_2=1$ and the only constraint in this category has index
$i$, set the candidate non-redundant constraint set $C_2= \{ i \}$
\item[] Else if $L_2>1$
\begin{itemize}
\item[] Sort $\{ \frac{1}{A_{i1}} \}$ in ascending 
order to get $\s_1= \{ \frac{1}{a_{i1}} \}$. Let 
$\s_2=\{  \lvert \frac{a_{i1}}{a_{i2}} \rvert  \}$ be 
obtained by re-arranging 
$\{  \lvert \frac{A_{i1}}{A_{i2}} \rvert \}$
in the same order as $\s_1$.
Denote $\S= [\s_1, \s_2]$ and set $j=1$.
Let $L$ be the number of rows of $\S$ and initial 
candidate non-redundant constraint set $C_2$ include
the indexes of all rows in $\S$.
\item[] While $j<L$
\begin{itemize}
\item[] Remove all rows in $\S$ that meet the condition 
$\S(i,2)<\S(j,2)$ and all corresponding indexes $i$ from $C_2$.
\item[] Let  $L$ be the number of rows of the reduced matrix $\S$
and set $j=j+1$.
\end{itemize}
\item[] End (While)
\end{itemize}
\item[] End (If)
\end{itemize}

{\it Category 3: $A_{i1}<0$ and $A_{i2}>0$}

For constraints in this category, the non-redundant
constraints are selected as follows (see Figure \ref{case3}):
\begin{figure}[htb]
\centerline{\includegraphics[height=6cm,width=6.5cm]
{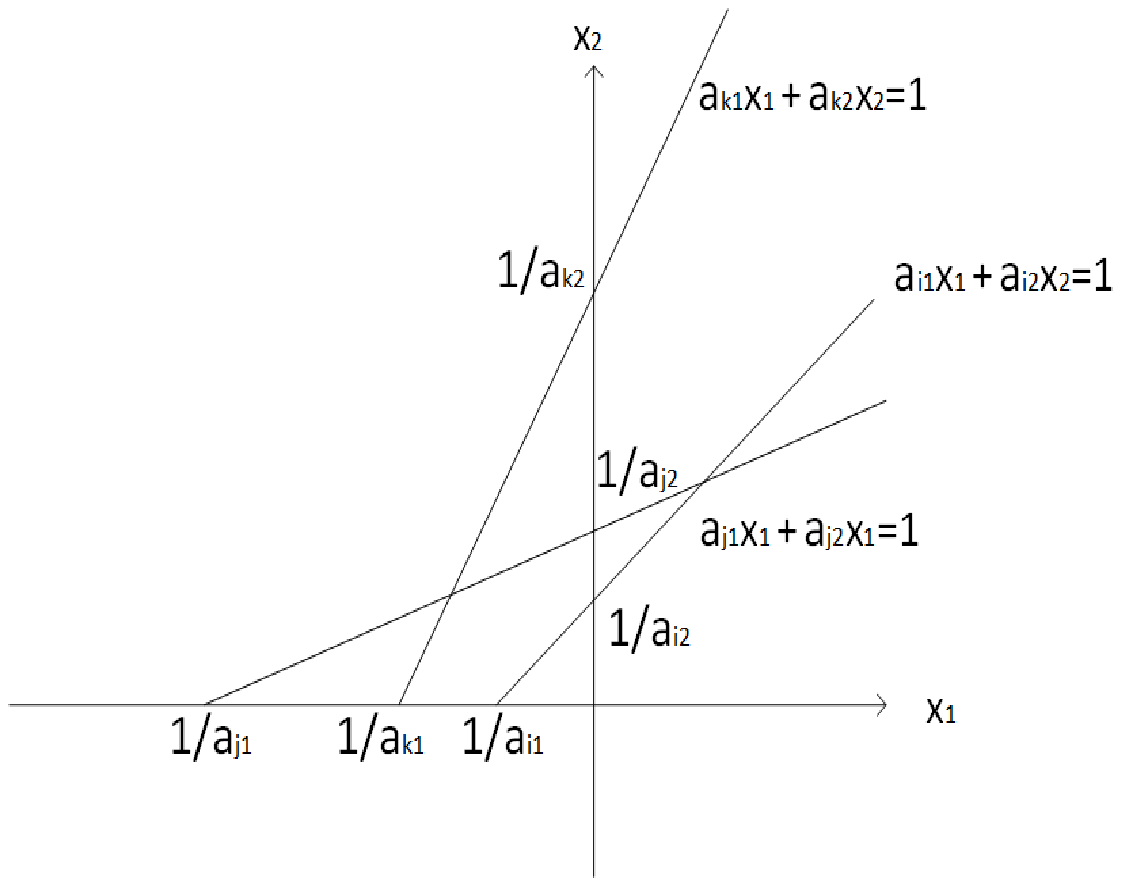}}
\caption{Constraints in Category 3}
\label{case3}
\end{figure}

\begin{itemize}
\item[] Let $L_3$ be the number of constraints in this
category.
\item[] If $L_3=1$ and the only constraint in this category has index
$i$, set the candidate non-redundant constraint set $C_3= \{ i \}$
\item[] Else if $L_3>1$
\begin{itemize}
\item[] Sort $\{ \frac{1}{A_{i2}} \}$ in ascending 
order to get $\s_1= \{ \frac{1}{a_{i2}} \}$. Let 
$\s_2=\{  \lvert \frac{a_{i2}}{a_{i1}} \rvert  \}$ be
obtained by re-arranging 
$\{  \lvert \frac{A_{i2}}{A_{i1}} \rvert \}$
in the same order as $\s_1$.
Denote $\S= [\s_1, \s_2]$ and set $j=1$.
Let $L$ be the number of rows of $\S$ and 
initial candidate non-redundant constraint
set $C_3$ include the indexes of all rows in $\S$.
\item[] While $j<L$
\begin{itemize}
\item[] Remove all rows in $\S$ that meet the condition 
$\S(i,2)>\S(j,2)$ and all corresponding indexes $i$ from $C_3$.
\item[] Let  $L$ be the number of rows of the reduced matrix $\S$
and set $j=j+1$.
\end{itemize}
\item[] End (While)
\end{itemize}
\item[] End (If)
\end{itemize}

{\it Category 4: $A_{i1}>0$ and $A_{i2}=0$}

For constraints in this category, we remove all
constraints except the $\imath$th constraints satisfying
$\frac{1}{A_{\imath 1}} =  \min_i \frac{1}{A_{i1}}$.

{\it Category 5: $A_{i1}=0$ and $A_{i2}>0$}

For constraints in this category, we remove all
constraints except the $\imath$th constraints satisfying
$\frac{1}{A_{\imath 2}} = \min_i \frac{1}{A_{i2}}$.

After removing the redundant constraints as described
as above, the number of rows in $\bar{\A}_1$ 
will be significantly reduced, hence
the number of equations expressed in the form
of (\ref{subEq}) will be significantly reduced.

Both Algorithm \ref{mainAlg} and Dantzig's pivot algorithm
are implemented in Matlab. Numerical test is carried 
out for randomly generated problems
which are obtained as follows: first, given the problem
size $m$, a matrix $\M$ with random entries of dimension
$m \times m$ and an identity matrix of dimension $m$
are generated. $\A=[\M~~~ \I]$ is determined whose
initial basic solution is composed of
the last $m$ columns. Then a positive vector $\b$
with random entries of dimension $m$ and a vector 
$\c=(\c_1, \0)$ with $\c_1$ a random vector of 
dimension $m$ are generated. For each of these 
LP problems, Dantzig's pivot algorithm and the 
double-pivot algorithm are used to solve the LP problem.
For each given problem size $m$, this test is repeated for 
$100$ randomly generated problems. The average iteration
number and average computational time in seconds are 
obtained. The test results are presented in Table \ref{table2},
It is easy to see that for all problems with different size, 
the double-pivot algorithm uses few iterations 
on average than Dantzig's pivot algorithm. For small size 
problems, Dantzig's pivot algorithm uses significant less 
CPU time than the double-pivot algorithm. As the problem
size increases, the double-pivot algorithm becomes more and 
more competitive to Dantzig's pivot algorithm. For $m=1000$,
the CPU times used by the two algorithms are very close.
For $m=2000$, it takes hours of the CPU times for either 
algorithm to solve a randomly generated dense LP 
problem\footnote{For $m=n/2$, the matrix $\M$ has 
$\frac{n^2}{4}$ non-zeros, but for most Netlib 
problems, there are only $\mathcal{O}(n)$ non-zeros. 
Therefore, it is not a surprise that solving the randomly 
generated {\bf dense} LP problems is time-consuming.}. 
Therefore, the test stops for problems 
with $m=1000$. According to the trends of the CPU 
times (and iteration numbers) used by the two 
algorithms for different problem sizes, we guess that 
for problems with size $m=10000$ and larger, the 
double-pivot algorithm will be more efficient than 
Dantzig's pivot algorithm. 


\begin{longtable}{|c|c|c|c|c|}
\hline          
\multirow{2}{*}{Problem} & \multicolumn{2}{|c|}
{Dantzig pivot rule}  
& \multicolumn{2}{|c|}{double pivot rule}   
\\ \cline{2-5}
{size m}  & iteration & CPU time (s) & iteration & CPU time (s)  \\ 
\hline
10      & 6.3800  & 0.0007   &   4.2200   &  0.0110    \\ \hline
100    & 160.03  & 0.0579   &   155.76   &  0.2537    \\ \hline
1000    & 17683  & 1641.1   &   7512  &  3097.7    \\ \hline
\caption{Comparison test for Dantzig pivot and double pivot rules}
\label{table2}
\end{longtable}

%

\section{Conclusion}

In this paper, a double-pivot simplex method is proposed.
Two upper bounds of the iteration numbers for the proposed
algorithm are derived. The first bound is very tight and 
attainable. The second bound, when it is applied to some special 
linear programming problems, such as LP with a totally 
unimodular matrix and Markov Decision Problem with a fixed 
discount rate, shows that the double pivot algorithm will
find the optimal solution in a strongly polynomial time.
The numerical test shows very
promising result. It is hoped that the double-pivot strategy
may lead to some strongly polynomial algorithms for general
linear programming problems.


\section{Acknowledgment}

This author would like to thank Dr. F. Vitor for sharing
his recent paper which is useful in comparing the proposed
work and his excellent work.

\section{Conflict of interest}

On behalf of all authors, the corresponding author states that 
there is no conflict of interest.

\end{document}